\documentclass[11pt]{article}

\usepackage[margin=1in]{geometry}
\usepackage[T1]{fontenc}
\usepackage[utf8]{inputenc}
\usepackage{lmodern}
\usepackage{microtype}
\usepackage{amsmath,amssymb,amsfonts,amsthm,mathtools}
\usepackage{bm}
\usepackage{enumitem}
\usepackage{graphicx}
\usepackage{subcaption}
\usepackage{xcolor}
\usepackage{hyperref}
\usepackage[nameinlink,capitalize,noabbrev]{cleveref}

\hypersetup{
  colorlinks=true,
  linkcolor=blue!60!black,
  citecolor=blue!60!black,
  urlcolor=blue!60!black,
  pdftitle={Title of the Article},
  pdfauthor={Sachin Bhalekar}
}

\theoremstyle{plain}
\newtheorem{theorem}{Theorem}[section]

\theoremstyle{definition}

\numberwithin{equation}{section}

\title{Stability Analysis of Pantograph Delay Differential Equations}
\author{
  Sachin Bhalekar\\
  School of Mathematics and Statistics, University of Hyderabad, Hyderabad, India\\
  \href{mailto:sachinbhalekar@uohyd.ac.in}{sachinbhalekar@uohyd.ac.in}
}
\date{}

\begin{document}

\maketitle
\begin{abstract}
This article investigates the stability of pantograph delay differential equations, in which the delayed argument is proportional to the present time. We derive analytic criteria that partition the parameter plane into unstable, asymptotically stable, and delay-dependent stability regions. The theoretical results are supported by numerical simulations that illustrate the sharpness of the stability boundaries. We also formulate a proportional-delay analogue of the Mackey--Glass chaotic delay differential equation and examine the resulting dynamical behaviour.
\end{abstract}

\noindent\textbf{Keywords.} Proportional delay; Stability analysis; Chaos.

\medskip

\noindent\textbf{2020 Mathematics Subject Classification.} Primary 34K20; Secondary 34K06, 37C75, 37D45.

\section{Introduction}
Pantograph differential equations arise naturally in engineering and applied science, especially in problems where the delay is proportional to the present time. In this article, we consider the scalar proportional-delay equation
\begin{equation}
\dot{x}(t)=ax(t)+bx(qt), \quad 0<q<1, \label{peq}
\end{equation}
where the term $x(qt)=x(t-(1-q)t)$ represents a time-dependent delay. As shown in \cite{BhalekarPatade2017}, equation~(\ref{peq}) admits the power-series representation
\begin{equation}
  x(t)=x(0)\sum_{n=0}^{\infty}\frac{t^n}{n!}\prod_{j=0}^{n-1}\left(a+bq^j\right). \label{ssol}
\end{equation}
In particular, when $a>0$ and $b>0$, the corresponding solution satisfies suitable exponential bounds.

As noted in the monograph of Bellen and Zennaro \cite{BellenZennaro2003}, pantograph-type equations occur in a surprisingly wide range of pure and applied contexts. Besides their original engineering motivation, such equations and closely related dilation functional equations appear in number theory \cite{Mahler1940}, stochastic models and probability-inspired problems \cite{Ferguson1972,ZhangLi2016}, nonlinear dynamics, quantum mechanics, and electrodynamics \cite{ZhangLi2016,DehghanShakeri2008}. They also arise in biological cell-growth models \cite{HallWake1989} and astrophysical modelling \cite{Ambartsumyan1944}, illustrating that proportional delay is not merely a technical variation of constant delay but a structure shared by several independent modelling problems.

The name ``pantograph equation'' is closely connected with the modelling of current collection systems in electric locomotives \cite{OckendonTayler1971}. A rigorous early analysis of the scalar functional differential equation (\ref{peq}) was given by Kato and McLeod \cite{KatoMcLeod1971}, who established foundational existence, uniqueness, and asymptotic properties. Subsequent work extended the theory to more general proportional-delay and pantograph-type equations, including generalized pantograph functional differential equations \cite{Iserles1993}. Numerical aspects have also received substantial attention because proportional delay prevents the direct use of many standard methods for ordinary and constant-delay differential equations; important contributions include studies of evolutionary delay equations \cite{BakerPaulWille1995}, pantograph integro-differential equations \cite{IserlesLiu1994}, and general numerical frameworks for delay differential equations \cite{BellenZennaro2003}. 

Equations of the form (\ref{peq}) are difficult to analyse by classical techniques. Unlike constant-delay differential equations, they do not lead to a simple characteristic equation, and the Laplace transform is not directly applicable for stability analysis. Standard numerical methods for ordinary differential equations, such as predictor-corrector and Runge-Kutta schemes, also require suitable modifications to handle the proportional delay. Although (\ref{peq}) admits an exact series solution, this solution is generally not expressible in terms of familiar special functions. Consequently, stability analysis based only on the solution series has significant limitations. These difficulties motivate the present study of the stability behaviour of (\ref{peq}).
The paper is organized as follows. Section \ref{stab} provides sufficient conditions for instability and stability by using the series solution, Lyapunov--Krasovskii theory, and Lyapunov--Chetaev theory. Numerical observations are also presented in this section. The $q$-analogue of the Mackey--Glass equation is studied in Section \ref{qmgs}, where chaos and bifurcations are investigated. The conclusions are summarized in Section \ref{concl}. The details of the numerical method are given in the appendix. The datasets related to this work are available in the GitHub repository \cite{BhalekarPantographGitHub}.

\section{Stability analysis}\label{stab}
\subsection{Challenges in applying classical theory}
For constant-delay equations, stability is often studied by substituting an exponential trial solution and deriving a characteristic quasi-polynomial. In the proportional-delay equation (\ref{peq}), however, the substitution $x(t)=e^{\lambda t}$ gives
\[
\lambda e^{\lambda t}=a e^{\lambda t}+b e^{\lambda qt}.
\]
The two exponential terms have different time scalings, so the factor $e^{\lambda t}$ cannot be eliminated to obtain an equation depending only on $\lambda$. More generally, the explicit dependence on both $t$ and $qt$ prevents a direct reduction to a characteristic equation. This is the main obstruction to applying the standard spectral stability theory used for ordinary or constant-delay differential equations.

\subsection{Some sufficient conditions for instability}
We first use the series representation (\ref{ssol}) to obtain elementary sufficient conditions for instability. Since equation (\ref{peq}) is linear, it is enough to exhibit one nonzero initial value that produces an unbounded solution; for convenience, we take $x(0)=1$.
\begin{theorem}\label{thm1}
If $a>0$ and $b>0$ then the system (\ref{peq}) is unstable for all $q\in (0, 1)$.
\end{theorem}
\textbf{Proof}: We have
\begin{equation}
0<a<a+b q^j, \quad j=0, 1, 2, \cdots.
\end{equation}
It follows that
\begin{equation}
0<a^n<\prod_{j=0}^{n-1}\left(a+bq^j\right).
\end{equation}
Consequently, by (\ref{ssol}),
\begin{equation}
e^{at}<x(t).
\end{equation}
Since $a>0$, the lower bound $e^{at}$ grows without bound as $t\to\infty$. Hence the solution of (\ref{peq}) with $x(0)=1$ is unbounded, and the zero solution is unstable.
\qed

\begin{theorem}\label{thm2}
If $a>-b>0$ then the system (\ref{peq}) is unstable for all $q\in (0, 1)$.
\end{theorem}

\textbf{Proof}: We have $b<0$ and $0<q<1$.

 Hence, $b<bq^j$ for $j=0, 1, 2, \cdots.$

 It follows that $0<a+b<a+bq^j$ for $j=0, 1, 2, \cdots.$

 Thus, $0<(a+b)^n<\prod_{j=0}^{n-1}\left(a+bq^j\right).$

 Hence, $e^{(a+b)t}<x(t)$, where $a+b>0$. Thus the solution with $x(0)=1$ is unbounded, which proves instability under the assumption $a>-b>0$.
 \qed 

\subsection{Lyapunov--Krasovskii theory for the stability}
In this section, we use Lyapunov--Krasovskii theory \cite{krasovskii1963stability} to provide a sufficient condition for stability of the system (\ref{peq}).

\begin{theorem}\label{LyTh}
The system (\ref{peq}) is asymptotically stable if $a<0, a<b<-a$ and $\frac{b^2}{a^2}<q<1$.
\end{theorem}
\textbf{Proof}: A Lyapunov--Krasovskii functional can be constructed  as follows:
\[
V(x(t))=x(t)^2+\mu\int_{qt}^{t} x(s)^2\,ds,
\]
where \(\mu>0\) is a constant to be determined.

Using Leibniz's rule,
\[
\frac{d}{dt}\int_{qt}^{t}x(s)^2\,ds
=
x(t)^2-q\,x(qt)^2.
\]

Along the solution trajectories,
\[
\dot V(x(t))
=
2x(t)\dot x(t)
+
\mu\bigl(x(t)^2-q\,x(qt)^2\bigr).
\]

Using (\ref{peq}),
\[
\dot V(x(t))
=
(2a+\mu)x(t)^2
+
2b\,x(t)x(qt)
-
\mu q\,x(qt)^2.
\]

We represent this in a quadratic form by defining
\[
z(t)=
\begin{bmatrix}
x(t)\\[2mm]
x(qt)
\end{bmatrix}.
\]
Then
\[
\dot V(z(t))=z(t)^T M z(t),
\]
where
\[
M=
\begin{bmatrix}
2a+\mu & b\\
b & -\mu q
\end{bmatrix}.
\]

If \(M\) is negative definite, then \(\dot V(z(t))<0\) for all nonzero \(z(t)\), implying asymptotic stability of the trivial solution of (\ref{peq}).

The symmetric matrix $M$ is negative definite if there exists \(\mu>0\) satisfying \(2a+\mu<0\), and \(\det(M)>0\).

That is,
\[
 a<-\mu/2<0\qquad  \text{and} \qquad (2a+\mu)(-\mu q)>b^2 
\]

Then \(V\) is a strict Lyapunov functional and the zero solution of (\ref{peq}) is asymptotically stable.

The quadratic function \(-\mu q(2a+\mu)\) attains its maximum at \(\mu=-a\), giving the sufficient condition
\[
a<0,
\qquad
b^2<a^2 q.
\]
Equivalently, 
$a<0, \qquad a<b<-a \qquad \text{and} \qquad \frac{b^2}{a^2}<q<1$ gives asymptotic stability of the system (\ref{peq}).
\qed 

\textbf{Note}: This result is a sufficient but not a necessary condition for stability. In fact, the numerical computations show that the system is asymptotically stable for all $0<q<1$ in the region $a<0, \quad a<b<-a$. We provide more details in Section \ref{conjec}. 

\subsection{Lyapunov--Chetaev theory for the instability}
In this section, we use Lyapunov--Chetaev theory (see \cite{Chetaev1961} and Section 9.8 in \cite{michel2008stability}) to provide a sufficient condition for instability of the system (\ref{peq}).

\begin{theorem}\label{CheTh}
The system (\ref{peq}) is unstable for all $q\in(0,1)$ if $a+b>0$.
\end{theorem}
Proof: We define the Lyapunov--Chetaev function as 
$$ V(x)=x^2.$$
Then $V(0)=0$ and $V(x_0)>0,$ for $x_0\ne 0$. 

Now, 
$$ \dot V(x(t))=2x(t)\dot x(t)=2ax(t)^2+2bx(t)x(qt).$$

Now, 
$$ \dot V(x_0)=\dot V(x(t))\vert_{t=0}=2(a+b)x_0^2>0.$$
Let us denote the solution of (\ref{peq}) with $x(0)=x_0$ by $x(t)$.
Hence, there exists $\epsilon>0$ such that 
$$\dot V(x(t))>0, \quad \text{for}\quad  0<t<\epsilon.$$
Thus, $V$ increases along the solution trajectories of (\ref{peq}) for sufficiently small values of $t$. Consequently, trajectories starting arbitrarily close to the origin move away from the origin and the system is unstable.
\qed

\subsection{Numerical observations} \label{conjec}
\subsubsection{Validation of results}
In this section, we use the predictor-corrector method described in the appendix to validate the results provided in the previous sections.
\begin{itemize}
  \item For $a>0$ and $a+b>0$, we observe that the solution is monotonically increasing (respectively, decreasing) for $x(0)>0$ (respectively, $x(0)<0$) and becomes unbounded quickly.\\
  Figures \ref{f01} and \ref{f02} show the diverging trajectories for $a=2.5, b=1, q=0.2, x(0)=1$ and $a=2.5, b=-2, q=0.7, x(0)=-1$, respectively.
  \item For $a<0$ and $a+b>0$, \\(i) $q=-a/b$ produces a solution trajectory resembling a straight line with slope 1 (see Figure \ref{f03}).\\
  (ii) $0<q<-a/b$ produces a solution trajectory that diverges very slowly, with logarithmic-like growth (see Figure \ref{f04}).\\
  (iii) $-a/b<q<1$ gives a monotonically diverging trajectory (see Figure \ref{f05}).
  \item For $a<0$, $a< b<-a$, and $\frac{b^2}{a^2}<q<1$ the solution trajectory converges monotonically to $0$ (see Figure \ref{f06}).
\end{itemize}

\begin{figure}[htbp]
  \centering
  \begin{subfigure}{0.48\textwidth}
    \centering
    \includegraphics[width=\textwidth]{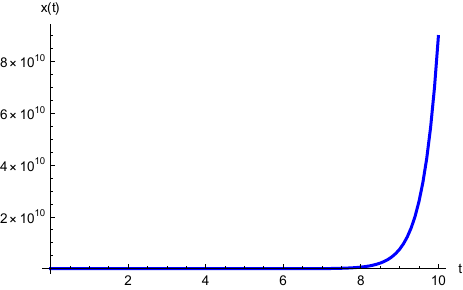}
    \caption{$a=2.5, b=1, q=0.2, x(0)=1$.}\label{f01}
  \end{subfigure}\hfill
  \begin{subfigure}{0.48\textwidth}
    \centering
    \includegraphics[width=\textwidth]{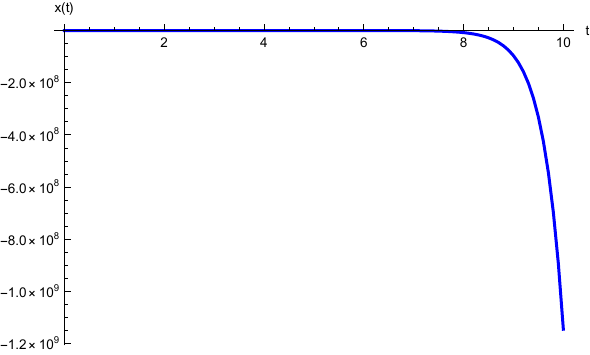}
    \caption{$a=2.5, b=-2, q=0.7, x(0)=-1$.}\label{f02}
  \end{subfigure}

  \begin{subfigure}{0.48\textwidth}
    \centering
    \includegraphics[width=\textwidth]{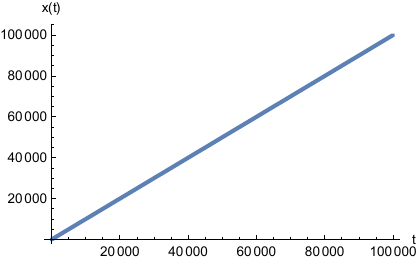}
    \caption{$a=-4, b=5, q=0.8=-a/b$.} \label{f03}
  \end{subfigure}\hfill
  \begin{subfigure}{0.48\textwidth}
    \centering
    \includegraphics[width=\textwidth]{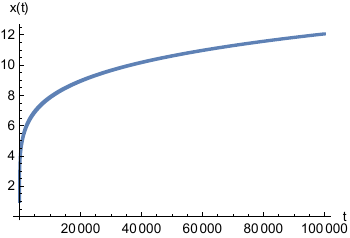}
    \caption{$a=-4, b=5, q=0.3<-a/b$.} \label{f04}
  \end{subfigure}

  \begin{subfigure}{0.48\textwidth}
    \centering
    \includegraphics[width=\textwidth]{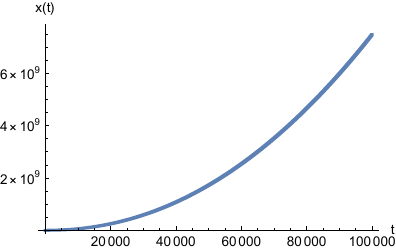}
    \caption{$a=-4, b=5, q=0.9>-a/b$.} \label{f05}
  \end{subfigure}\hfill
  \begin{subfigure}{0.48\textwidth}
    \centering
    \includegraphics[width=\textwidth]{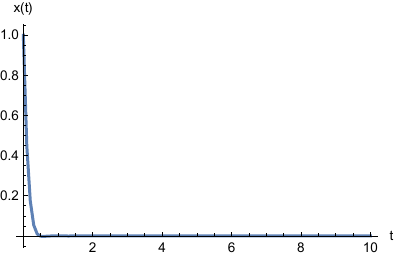}
    \caption{$a=-6, b=-5, q=0.8>b^2/a^2$.} \label{f06}
  \end{subfigure}
  \caption{Numerical solution trajectories for representative parameter values.}
  \label{fig:numerical-observations}
\end{figure}

\subsubsection{Remaining regions in the parameter plane}
In this section, we consider the regions in the $ab$-plane for which no theoretical result was provided above. 

\textbf{Region 1:} We consider $a<0$, $a<b<-a$, and  $0<q\leq \frac{b^2}{a^2}$. We observe that the system is asymptotically stable in this region as well. However, the convergence is ultra-slow for smaller values of $q$. For $b\in(a,0)$, the solution exhibits damped oscillations and converges to $0$ (see Figure \ref{f07}). On the other hand, for $b\in(0,-a)$, the solution converges monotonically to $0$ (see Figure \ref{f08}).

\textbf{Region 2:} We consider $b<0$ and $b<a<-b$. We observe that, for each $a<0$, there exists $b_*<a$ such that, when $b\in(b_*, a)$, the system is stable for larger values of $q\in(0,1)$ (see Figure \ref{f09}) and unstable for smaller values of $q$ (see Figure \ref{f10}). In the unstable case, the trajectory oscillates and diverges ultra-slowly.\\
 If $b<b_*$, then the system is unstable for all $q\in(0,1)$ (see Figure \ref{f11}). The trajectory oscillates, and the divergence is faster for $q$ near $1$.\\ 
 We list a few values of the pairs $(a, b_*)$ in Table \ref{tab1}. The data are well approximated by $b_*=0.014455 a^2+1.391065a-0.03994898$. 
If $a>0$ in this region, then the solution is unstable for any $q\in(0,1)$ (see Figure \ref{f12}). The trajectory tends to $-\text{sgn}(x(0))\infty$ very quickly. We get faster divergence as $q$ tends to $0$.

\begin{figure}[htbp]
  \centering
  \begin{subfigure}{0.48\textwidth}
    \centering
    \includegraphics[width=\textwidth]{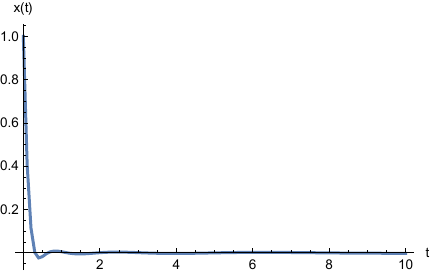}
    \caption{$a=-6, b=-5, b^2/a^2=0.6944, q=0.65<b^2/a^2$.}\label{f07}
  \end{subfigure}\hfill
  \begin{subfigure}{0.48\textwidth}
    \centering
    \includegraphics[width=\textwidth]{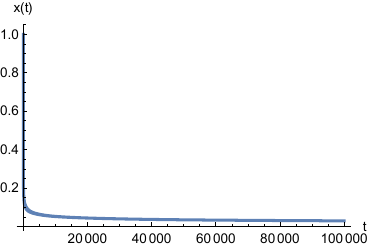}
    \caption{$a=-6, b=4, q=0.2<b^2/a^2$.}\label{f08}
  \end{subfigure}

  \begin{subfigure}{0.48\textwidth}
    \centering
    \includegraphics[width=\textwidth]{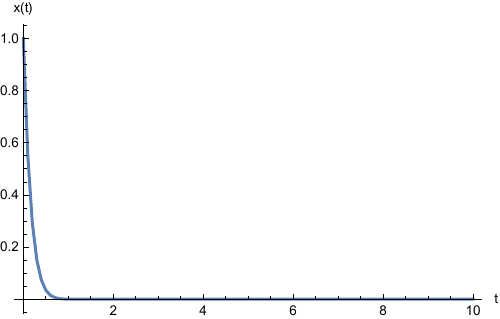}
    \caption{$a=-3, b_*=-4.1, b=-3.5\in(b_*, a), q=0.9$.} \label{f09}
  \end{subfigure}\hfill
  \begin{subfigure}{0.48\textwidth}
    \centering
    \includegraphics[width=\textwidth]{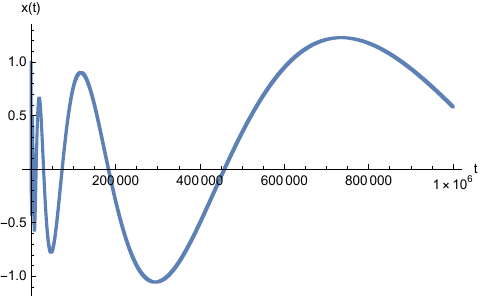}
    \caption{$a=-3, b_*=-4.1, b=-3.5\in(b_*, a), q=0.4$.} \label{f10}
  \end{subfigure}

  \begin{subfigure}{0.48\textwidth}
    \centering
    \includegraphics[width=\textwidth]{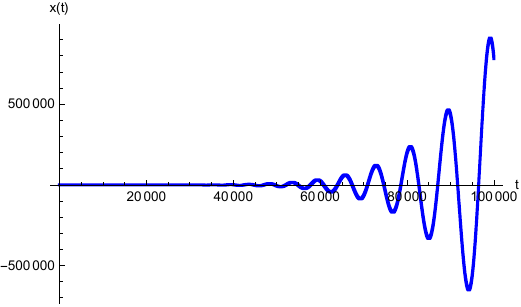}
    \caption{$a=-3, b_*=-4.1, b=-4.2<b_*, q=0.95$.} \label{f11}
  \end{subfigure}\hfill
  \begin{subfigure}{0.48\textwidth}
    \centering
    \includegraphics[width=\textwidth]{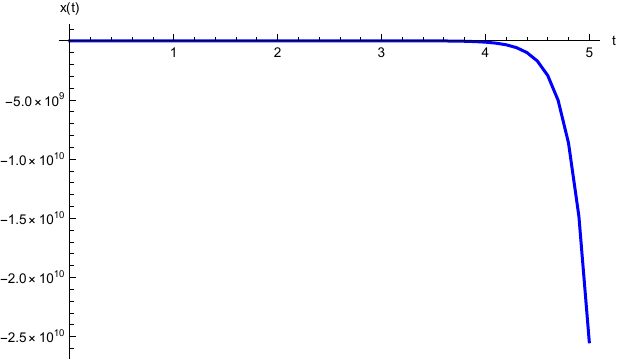}
    \caption{$b=-7, a=5.6\in(0, -b), q=0.5, x(0)=1$.} \label{f12}
  \end{subfigure}
  \caption{Numerical solutions in the remaining regions.}
  \label{fig:numerical-observations2}
\end{figure}

\begin{table}[h!]
\centering
\begin{tabular}{|c|c|c|c|c|c|c|c|}
\hline
$a$ & $b_*$ & $a$ & $b_*$ & $a$ & $b_*$ & $a$ & $b_*$ \\
\hline
0 & 0 & -2.6 & -3.6 & -5.2 & -6.9 & -7.8 & -10 \\
-0.2 & -0.3 & -2.8 & -3.8 & -5.4 & -7.1 & -8 & -10.2 \\
-0.4 & -0.6 & -3 & -4.1 & -5.6 & -7.4 & -8.2 & -10.5 \\
-0.6 & -0.9 & -3.2 & -4.3 & -5.8 & -7.6 & -8.4 & -10.7 \\
-0.8 & -1.1 & -3.4 & -4.6 & -6 & -7.9 & -8.6 & -10.9 \\
-1 & -1.4 & -3.6 & -4.9 & -6.2 & -8.1 & -8.8 & -11.2 \\
-1.2 & -1.7 & -3.8 & -5.1 & -6.4 & -8.4 & -9 & -11.4 \\
-1.4 & -2 & -4 & -5.4 & -6.6 & -8.6 & -9.2 & -11.6 \\
-1.6 & -2.2 & -4.2 & -5.6 & -6.8 & -8.8 & -9.4 & -11.8 \\
-1.8 & -2.5 & -4.4 & -5.9 & -7 & -9.1 & -9.6 & -12.1 \\
-2 & -2.8 & -4.6 & -6.1 & -7.2 & -9.3 & -9.8 & -12.3 \\
-2.2 & -3 & -4.8 & -6.4 & -7.4 & -9.5 & -10 & -12.5 \\
-2.4 & -3.3 & -5 & -6.6 & -7.6 & -9.8 & -10.2 & -12.7 \\
\hline
\end{tabular}
\caption{Data table of $a$ and $b_*$ values.}
\label{tab1}
\end{table}

\subsection{The stability diagram in parameter plane}
Based on the results provided in previous sections and the observations based on numerical computations, we sketch the stability regions in the parameter $ab$-plane in Figure~\ref{fig1}.

\begin{figure}[htbp]
  \centering
  \includegraphics[width=0.8\textwidth]{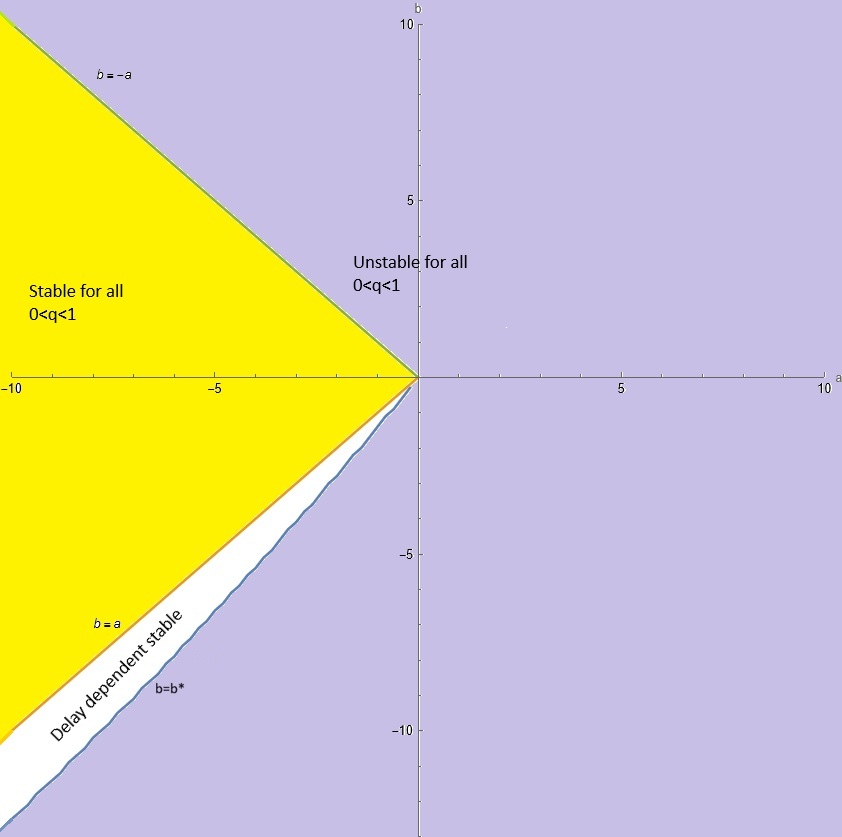}
  \caption{Various stability regions of system (\ref{peq}) in the parameter plane.}
  \label{fig1}
\end{figure}


\section{Bifurcations and chaos in a $q$-analogue of the Mackey--Glass system} \label{qmgs}
The classical Mackey--Glass chaotic system \cite{mackey1977oscillation,glass2010mackey} is given by
\begin{equation*}
\dot x(t)=-\beta x(t)+\frac{\alpha x(t-\tau)}{1+x(t-\tau)^c}, 
\end{equation*}
where $\tau>0$ is the (constant) delay.

We define the $q$-analogue of the Mackey--Glass system by 
\begin{equation}
\dot x(t)=-\beta x(t)+\frac{\alpha x(qt)}{1+x(qt)^c}, \label{mg}
\end{equation}
where $\alpha>\beta>0,$ $c>2$ and $0<q<1$.

The system (\ref{mg}) admits the equilibrium points $x_1^*=0$ and $x_2^*=\left(\frac{\alpha-\beta}{\beta}\right)^{1/c}$. Note that, $x_2^*$ is a collection of equilibrium points with identical stability properties.

Linearization of (\ref{mg}) at $x_1^*$ is given by
\begin{equation*}
  \dot x(t)= -\beta x(t)+\alpha x(qt).
\end{equation*}
Comparing this with (\ref{peq}), we get $a=-\beta$ and $b=\alpha$. Our assumptions show that $a+b>0$. Hence, by Theorem \ref{CheTh}, the equilibrium $x_1^*$ is unstable for all $q\in(0, 1)$.

\begin{theorem}\label{mgstab}
For system (\ref{mg}), the relevant bifurcation values of the parameter $\alpha$ are
\[
\alpha_1^*=\frac{\beta c}{c-2}
\qquad \text{and} \qquad
\alpha_2^*=\frac{\beta^2 c}{b_*+\beta(c-1)},
\]
where $b_*$ is the threshold value described in Table \ref{tab1}.

The nonzero equilibria $x_2^*$ satisfy the following stability classification:
\begin{itemize}
  \item they are asymptotically stable for all $q\in(0, 1)$ if $\beta<\alpha<\alpha_1^*$;
  \item they are delay-dependent stable if $\alpha_1^*<\alpha<\alpha_2^*$;
  \item they are unstable for all $q\in(0, 1)$ if $\alpha>\alpha_2^*$.
\end{itemize}
\end{theorem}
\textbf{Proof}: Linearization of the equation (\ref{mg}) at $x_2^*$ is given by
\begin{equation*}
 \dot x(t)= -\beta x(t)+\left(\frac{(\beta-\alpha)\beta c}{\alpha}+\beta\right) x(qt).  
\end{equation*}
Comparison with (\ref{peq}) gives $a=-\beta$ and $b=\frac{(\beta-\alpha)\beta c}{\alpha}+\beta$.

Our assumptions on $\alpha$, $\beta$, and $c$ show that $a<0$ and $a+b<0$. Following the results discussed in Section \ref{stab} and summarized in Figure \ref{fig1}, we conclude that the nonzero equilibria $x_2^*$ are asymptotically stable for all $0<q<1$ if $a<b<-a$, delay-dependent stable if $b_*<b<a$, and unstable if $b<b_*$.

Thus, the first bifurcation point is obtained by solving $a=b$, namely,
\begin{equation*}
  -\beta=\frac{(\beta-\alpha)\beta c}{\alpha}+\beta.
\end{equation*}
On simplification, we get the first bifurcation value of the parameter $\alpha$ as
\begin{equation}
\alpha_1^*=\frac{\beta c}{c-2}.
\end{equation}
The second bifurcation value is obtained by solving $b=b_*$, namely,
\begin{equation*}
\frac{(\beta-\alpha)\beta c}{\alpha}+\beta=b_*.
\end{equation*}
This gives the second bifurcation value of the parameter $\alpha$ as
\begin{equation}
\alpha_2^*=\frac{\beta^2 c}{b_*+\beta(c-1)}.
\end{equation}
This proves the theorem.
\qed 

\subsection{Validation of Theorem \ref{mgstab} and investigation of chaos}
Let us fix $\beta=1$ and $c=10$ in the system (\ref{mg}). Then $a=-1$, and Table \ref{tab1} gives $b_*=-1.4$. Furthermore, the bifurcation values for the parameter $\alpha$ are $\alpha_1^*=1.25$ and $\alpha_2^*=1.31579$. Theorem \ref{mgstab} states that the nonzero equilibria $x_2^*$ satisfy the following stability classification:
\begin{itemize}
  \item they are asymptotically stable for all $q\in(0, 1)$ if $1<\alpha<1.25$;
  \item they are delay-dependent stable if $1.25<\alpha<1.31579$;
  \item they are unstable for all $q\in(0, 1)$ if $\alpha>1.31579$.
\end{itemize}

For $\alpha=1.24,$ the two nonzero equilibria are $0.867004$ and $-0.867004$. If $x(0)>0$, then the resulting solution converges to $0.867004$ (see Figure \ref{f13}). For $x(0)<0$, the solution converges to $-0.867004$ (see Figure \ref{f14}).

For $\alpha=1.28$, the equilibrium points $x_2^*$ have values $0.880473$ and $-0.880473$. We get the stable solution for $q=0.7$, as shown in Figure \ref{f15}. For $q=0.6$ and $x(0)=0.2$, the solution trajectory oscillates around $0.867004$ (see Figure \ref{f16}). The period of solution wave increases with $t$. It is expected because the delay $\tau(t)=(1-q)t$ increases with $t$.

For $\alpha=1.6$, the equilibrium points $x_2^*$ take the values $\pm 0.9502$. We observe chaotic oscillations for $q=0.9$ as shown in Figure \ref{f17}. The chaotic attractor is shown in Figure \ref{f18}. 
In fact, we have coexisting chaotic attractors as shown in Figure \ref{f18}.

\begin{figure}[htbp]
  \centering
  \begin{subfigure}{0.48\textwidth}
    \centering
    \includegraphics[width=\textwidth]{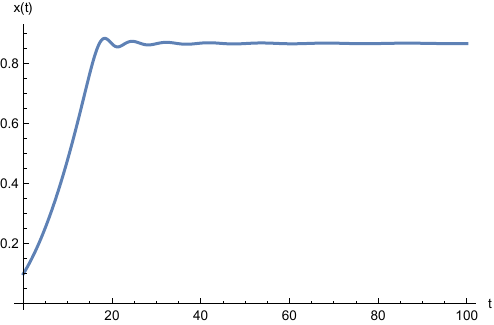}
    \caption{$\alpha=1.24$, $x(0)>0$: convergence to $0.867004$.}
    \label{f13}
  \end{subfigure}\hfill
  \begin{subfigure}{0.48\textwidth}
    \centering
    \includegraphics[width=\textwidth]{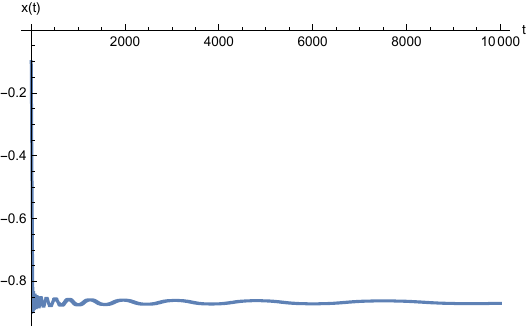}
    \caption{$\alpha=1.24$, $x(0)<0$: convergence to $-0.867004$.}
    \label{f14}
  \end{subfigure}

  \begin{subfigure}{0.48\textwidth}
    \centering
    \includegraphics[width=\textwidth]{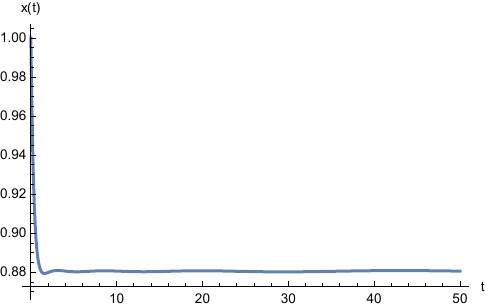}
    \caption{$\alpha=1.28$, $q=0.7$: stable solution.}
    \label{f15}
  \end{subfigure}\hfill
  \begin{subfigure}{0.48\textwidth}
    \centering
    \includegraphics[width=\textwidth]{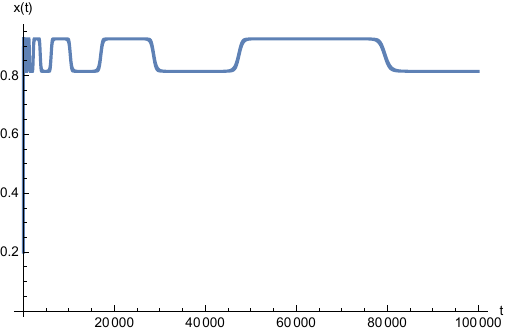}
    \caption{$\alpha=1.28$, $q=0.6$, $x(0)=0.2$: oscillation around $0.867004$.}
    \label{f16}
  \end{subfigure}

  \begin{subfigure}{0.48\textwidth}
    \centering
    \includegraphics[width=\textwidth]{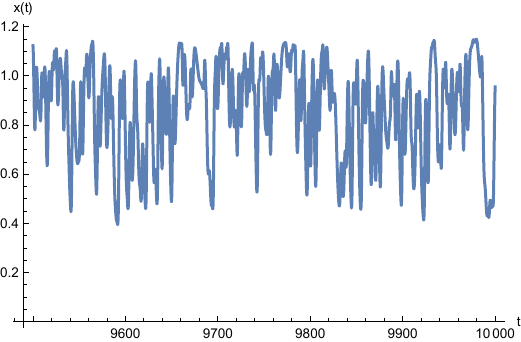}
    \caption{$\alpha=1.6$, $q=0.9$: chaotic oscillations.}
    \label{f17}
  \end{subfigure}\hfill
  \begin{subfigure}{0.48\textwidth}
    \centering
    \includegraphics[width=\textwidth]{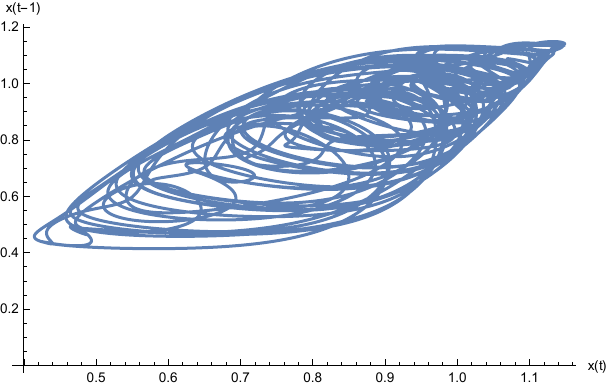}
    \caption{$\alpha=1.6$, $q=0.9$: chaotic attractor.}
    \label{f18}
  \end{subfigure}
  \caption{Validation of Theorem~\ref{mgstab} and numerical investigation of stability and chaos in the proportional-delay Mackey--Glass equation.}
  \label{fig:mg-validation-chaos}
\end{figure}

\begin{figure}[htbp]
  \centering
  \includegraphics[width=0.8\textwidth]{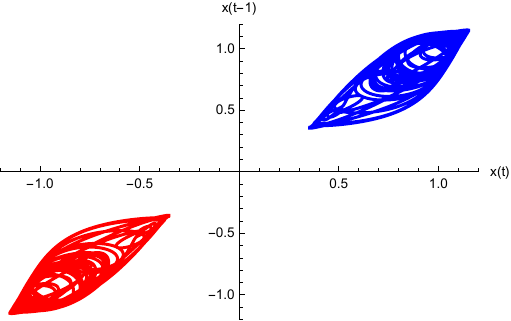}
  \caption{Coexisting chaotic attractors for the proportional-delay Mackey--Glass equation with $\alpha=1.6, q=0.7$.}
  \label{f19}
\end{figure}

\section{Conclusion and discussion} \label{concl}
This work examined the scalar pantograph delay equation
\[
\dot x(t)=a x(t)+b x(qt), \qquad 0<q<1,
\]
from the viewpoint of stability, bifurcation, and nonlinear dynamics. Since the proportional delay prevents the direct construction of a characteristic equation by the usual exponential ansatz, the analysis was based on a combination of exact series representations, Lyapunov methods, and numerical experiments. The series solution yielded simple sufficient conditions for instability, including the cases $a>0$, $b>0$ and $a>-b>0$. A Lyapunov--Krasovskii functional gave a rigorous sufficient condition for asymptotic stability when $a<0$, $a<b<-a$, and $b^2/a^2<q<1$, whereas a Lyapunov--Chetaev argument showed instability when $a+b>0$.

The numerical study complements these analytical criteria and gives a more complete picture of the parameter plane. The computations indicate that the condition obtained from the Lyapunov--Krasovskii functional is conservative: solutions appear to be asymptotically stable for all $0<q<1$ throughout the wider region $a<0$ and $a<b<-a$. In the region $a<0$ and $b<a$, the simulations suggest the existence of a threshold $b_*=b_*(a)$ separating parameters that are unstable for all proportional delays from parameters for which stability depends on $q$. For the sampled values of $a$, this threshold is well approximated by the quadratic fit
\[
b_*(a)=0.014455a^2+1.391065a-0.03994898.
\]
Thus, the stability diagram consists of delay-independent stable regions, delay-independent unstable regions, and an intermediate delay-dependent region. A notable numerical feature in several boundary and delay-dependent cases is ultra-slow convergence or divergence, often accompanied by oscillations whose effective period increases with time because the delay $\tau(t)=(1-q)t$ grows unboundedly.

The same stability framework was then applied to a proportional-delay analogue of the Mackey--Glass equation. Linearization about the nonzero equilibria reduces the local stability problem to the pantograph equation with coefficients determined by the biological parameters. This reduction gives explicit bifurcation thresholds
\[
\alpha_1^*=\frac{\beta c}{c-2},
\qquad
\alpha_2^*=\frac{\beta^2c}{b_*+\beta(c-1)},
\]
which separate delay-independent stability, delay-dependent stability, and instability. For the representative case $\beta=1$ and $c=10$, these values are $\alpha_1^*=1.25$ and $\alpha_2^*=1.31579$. The simulations validate the predicted stable and delay-dependent regimes and also reveal rich nonlinear behaviour, including sustained oscillations, chaotic trajectories, and coexisting chaotic attractors for larger values of $\alpha$.

Several questions remain open. First, the numerical conjecture that the entire region $a<0$, $a<b<-a$ is asymptotically stable for every $q\in(0,1)$ deserves a rigorous proof. Second, the threshold curve $b_*(a)$ was obtained empirically; deriving it analytically, or proving sharp bounds for it, would substantially strengthen the stability classification. These directions would help turn the present stability diagram and numerical observations into a more complete theory for pantograph-type delay systems.

\appendix

\setcounter{equation}{0}
\renewcommand{\theequation}{A\arabic{equation}}
\section*{Appendix: A predictor-corrector method to solve pantograph equation}
Consider
\begin{equation}
\dot x(t)=f\left(x(t), x(qt)\right), \qquad x(0)=x_0, \qquad t\in[0,T]. \label{1}
\end{equation}

Let us discretize the interval $[0,T]$ as
\[
0=t_0<t_1=h<t_2=2h<\cdots<t_N=Nh=T,
\]
where $h>0$. Assume we already know $x_i$, $i=0,1,\cdots,n$ and wish to find $x_{n+1}$.

Integrating (\ref{1}) on the interval $[t_n,t_{n+1}]$ and using the trapezoidal rule, we get

\begin{equation}
x_{n+1}
=
x_n+\frac{h}{2}
\left[
f(t_{n+1},x_{n+1},\mu_{n+1})
+
f(t_n,x_n,\mu_n)
\right],\label{3}
\end{equation}
where $x_n\approx x(t_n)$, and $\mu_n \text{ is an approximation to } x(qt_n)$.

Note that $qt_n<t_n$ as $0<q<1$.

If $qt_n=t_m$ for some $m<n$, then $\mu_n=x_m$.

Otherwise, $qt_n\in(t_{m-1},t_m)$.

In this case, we approximate $x(qt_n)$ by utilizing linear interpolation:
\[
x(qt_n)\approx \mu_n= (1-qn+m-1) x_{m-1}+(qn-m+1)x_m
\]

As $qn h \in ((m-1)h,mh)$, we have $[qn]=m-1$.

Hence,
\begin{eqnarray}
\mu_n&=&(1-qn+[qn])x_{m-1}
+
(qn-[qn])x_{m}
\nonumber\\
&=&
(1-\delta_n)x_{[qn]}
+
\delta_n x_{[qn]+1},\label{4}
\end{eqnarray}

where $\delta_n=qn-[qn]$ is the fractional part of $qn$.

Notice that $0\le \delta_n<1,$ and everything in the right-hand side of (\ref{4}) is known.

Now $x_{n+1}$ appears on both sides of (\ref{3}). One way to find it is by a predictor-corrector method. First, an approximate value of $x_{n+1}$ is obtained using a predictor, and then this approximation is refined by a correction step.

\textbf{Predictor Step:}

We use the rectangle rule on the interval $[t_n,t_{n+1}]$ and define
\begin{equation}
x_{n+1}^{p}
=
x_n+h\,f(t_n,x_n,\mu_n).\label{5}
\end{equation}

\textbf{Corrector Step:}

\begin{equation}
x_{n+1}
=
x_n+\frac{h}{2}
\left[
f(t_{n+1},x_{n+1}^{p},\mu_{n+1})
+
f(t_n,x_n,\mu_n)
\right].\label{6}
\end{equation}

Now $q(n+1)<n+1$.

\begin{itemize}
\item If $q(n+1)< n$, then
\begin{equation}
\mu_{n+1}
=
(1-\delta_{n+1})x_{[q(n+1)]}
+
\delta_{n+1}x_{[q(n+1)]+1},\label{7}
\end{equation}
where $\delta_{n+1}=q(n+1)-[q(n+1)]$.
In this case, $[q(n+1)]+1\le n+1$ and all the terms on the right-hand side of (\ref{7}) are already known.

\item If $n<q(n+1)<n+1,$ then $[q(n+1)]+1=n+1,$
and the unknown term $x_{n+1}$ appears in the right-hand side of (\ref{7}) . In this case, we replace $\mu_{n+1}$ by
\begin{equation}
\mu_{n+1}^{p}=(1-\delta_{n+1})x_n+\delta_{n+1}x_{n+1}^{p},\label{8}
\end{equation}
and use this predictor value in (\ref{6}). After computing (\ref{6}), we will correct the value of $\mu_{n+1}$ by utilizing (\ref{6}) and proceed to the next step.
\end{itemize}

\textbf{Illustration:}
We illustrate above algorithm and provide the details of the evaluation of few terms.

Given $x_0$, we have to find $x_1$.
\begin{equation*}
x_1^{p}
=
x_0+h\,f(0,x_0,\mu_0),
\end{equation*}
where $\mu_0=x(q\cdot 0)=x(0)=x_0$ and
\begin{equation}
x_1
=
x_0+\frac{h}{2}
\left[
f(h,x_1^{p},\mu_1)
+
f(0,x_0,\mu_0)
\right],\label{9}
\end{equation}
where
\begin{equation*}
\mu_1
=
(1-\delta_1)x_{[q]}
+
\delta_1 x_{[q]+1},
\end{equation*}
with $\delta_1=q-[q],$ $x_{[q]}=x_0$ and $x_{[q]+1}=x_1$ as $0<q<1$.

Thus,
\begin{equation}
\mu_1=(1-q)x_0+q\,x_1.\label{10}
\end{equation}

Since $x_1$ is not known yet, we take
\[
\mu_1^{p}=(1-q)x_0+q\,x_1^{p}
\]
and use this in (\ref{9}). Once we evaluate (\ref{9}), we update $\mu_1$ using (\ref{10}).

Now, $x_2^{p}=x_1+h\,f(h,x_1,\mu_1),$ and 
\begin{equation}
\mu_2=(1-\delta_2)x_{[2q]}+\delta_2 x_{[2q]+1},
\label{11}
\end{equation}
where
\[
[2q]
=
\begin{cases}
0, & 0<q<\dfrac12, \\[1ex]
1, & \dfrac12\le q<1.
\end{cases}
\]

Also,
\[
\delta_2
=
2q-[2q]
=
\begin{cases}
2q, & 0<q<\dfrac12, \\[1ex]
2q-1, & \dfrac12\le q<1.
\end{cases}
\]
Thus, if $0<q<\dfrac12$, then
\[
\mu_2=(1-2q)x_0+2q\,x_1.
\]

If $\dfrac12\le q<1$, then
\[
\mu_2=(2-2q)x_1+(2q-1)x_2.
\]

As $x_2$ is not yet available, we take
\[
\mu_2^{p}
=
(2-2q)x_1+(2q-1)x_2^{p},
\]
and then evaluate
\[
x_2
=
x_1+\frac{h}{2}
\left[
f(2h,x_2^{p},\mu_2)
+
f(h,x_1,\mu_1)
\right].
\]

Similarly,
\[
\mu_3
=
(1-\delta_3)x_{[3q]}
+
\delta_3 x_{[3q]+1}.
\]

If $0<q<\dfrac13$, then $\delta_3=3q,$ $[3q]=0$ and hence $\mu_3=(1-3q)x_0+3q\,x_1$.

If $\dfrac13\le q<\dfrac23$, then $\delta_3=3q-1,$ $[3q]=1$ and hence $\mu_3=(2-3q)x_1+(3q-1)x_2$.

If $\dfrac23\le q<1$, then $\delta_3=3q-2,$ $[3q]=2$. In this case, first we compute $x_3^{p}=x_2+h\,f(2h,x_2,\mu_2)$ and define $\mu_3^p=(3-3q)x_2+(3q-2)x_3^{p}$. The procedure will be continued.

In the algorithm, we need to find a natural number $k=\left[\frac{1}{1-q}\right]$.

The procedure will be as follows.

For $n=1,2,\ldots,k$:
\begin{enumerate}
\item Find $x_n^{p}$ using (\ref{5}).
\item Find $\mu_n^{p}$ using (\ref{8}).
\item Find $x_n$ using (\ref{6}).
\item Find $\mu_n$ using $x_n$.
\end{enumerate}

For $n>K$:
\begin{enumerate}
\item Find $x_n^{p}$ using (\ref{5}).
\item Find $\mu_n$ using (\ref{7}).
\item Find $x_n$ using (\ref{6}).
\end{enumerate}
The C program for this method is available at \cite{BhalekarPantographGitHub}.
More details on numerical methods to solve delay differential equations are available in \cite{BellenZennaro2003}.

\bibliographystyle{unsrt}
\bibliography{references}

\end{document}